\documentclass[12pt,a4paper]{amsart}

\usepackage{amsmath, amssymb, amsthm, times,color}
\usepackage[margin=1in]{geometry}
\usepackage{graphicx}
\usepackage{enumerate}
\usepackage[all, cmtip]{xy}

\usepackage[style=alphabetic,natbib=true, backend=bibtex, maxnames=99, maxalphanames=5]{biblatex}

\usepackage{hyperref}
\hypersetup{
    bookmarks=true,         
    unicode=false,          
    pdftoolbar=true,        
    pdfmenubar=true,        
    pdffitwindow=false,     
    pdfstartview={FitH},    
    pdftitle={My title},    
    pdfauthor={Author},     
    pdfsubject={Subject},   
    pdfcreator={Creator},   
    pdfproducer={Producer}, 
    pdfkeywords={keyword1} {key2} {key3}, 
    pdfnewwindow=true,      
    colorlinks=true,       
    linkcolor=red,          
    citecolor=green,        
    filecolor=magenta,      
    urlcolor=cyan           
}
\usepackage{cleveref}

 \newtheorem{introtheorem}{Theorem}
 
 \crefname{introtheorem}{Theorem}{Theorems}
 \Crefname{introtheorem}{Theorem}{Theorems}
  \newtheorem{introthm}[introtheorem]{Theorem}
   \crefname{introthm}{theorem}{theorems}
 \Crefname{introthm}{Theorem}{Theorems}

  \crefname{introcorollary}{Corollary}{Corollaries}
 \Crefname{introcorollary}{Corollary}{Corollaries}

 \newtheorem{introcor}[introtheorem]{Corollary}
   \crefname{introcor}{Corollary}{Corollaries}
 \Crefname{introcor}{Corollary}{Corollaries}
 
   \crefname{introconjecture}{Conjectures}{Conjectures}
 \Crefname{introconjecture}{Conjecture}{Conjectures}
 
    \crefname{introconj}{Conjectures}{Conjectures}
 \Crefname{introconj}{Conjecture}{Conjectures}

     \crefname{introlem}{Lemma}{Lemmas}
 \Crefname{introlem}{Lemma}{Lemmas}
 
 \crefname{introremark}{Remark}{Remarks}
 \Crefname{introremark}{Remark}{Remarks}
 
  \crefname{introrem}{Remark}{Remarks}
 \Crefname{introrem}{Remark}{Remarks}

   \crefname{introprop}{Proposition}{Propositions}
 \Crefname{introprop}{Proposition}{Propositions}

   \crefname{introdefn}{Definition}{Definitions}
 \Crefname{introdefn}{Definition}{Definitions}
 
   \crefname{intronotn}{Notation}{Notations}
 \Crefname{intronotn}{Notation}{Notations}
 
   \crefname{introtask}{Task}{Tasks}
 \Crefname{introtask}{Task}{Tasks}
 
  \crefname{introprob}{Problem}{Problems}
 \Crefname{introprob}{Problem}{Problems}
 
   \crefname{introquestion}{question}{questions}
 \Crefname{introquestion}{Question}{Questions}

 \crefname{theorem}{Theorem}{Theorems}
 \Crefname{theorem}{Theorem}{Theorems}
  \crefname{thm}{Theorem}{Theorems}
 \Crefname{thm}{Theorem}{Theorems}

  \crefname{corollary}{Corollary}{Corollaries}
 \Crefname{corollary}{Corollary}{Corollaries}

   \crefname{cor}{Corollary}{Corollaries}
 \Crefname{cor}{Corollary}{Corollaries}

   \crefname{conjecture}{Conjectures}{Conjectures}
 \Crefname{conjecture}{Conjecture}{Conjectures}

    \crefname{conj}{Conjectures}{Conjectures}
 \Crefname{conj}{Conjecture}{Conjectures}

     \crefname{lem}{Lemma}{Lemmas}
 \Crefname{lem}{Lemma}{Lemmas}

      \crefname{lemma}{Lemma}{Lemmas}
 \Crefname{lemma}{Lemma}{Lemmas}

 \crefname{remark}{remark}{remarks}
 \Crefname{remark}{Remark}{Remarks}

  \crefname{rem}{Remark}{Remarks}
 \Crefname{rem}{Remark}{Remarks}

   \crefname{rem}{Remark}{Remarks}
 \Crefname{rem}{Remark}{Remarks}

   \crefname{proposition}{Proposition}{Proposition}
 \Crefname{proposition}{Proposition}{Proposition}

    \crefname{prop}{Proposition}{Propositions}
 \Crefname{prop}{Proposition}{Propositions}

   \crefname{defn}{Definition}{Definitions}
 \Crefname{defn}{Definition}{Definitions}

   \crefname{notn}{Notation}{Notations}
 \Crefname{notn}{Notation}{Notations}

   \crefname{task}{Task}{Tasks}
 \Crefname{task}{Task}{Tasks}

  \crefname{prob}{Problem}{Problems}
 \Crefname{prob}{Problem}{Problems}

   \crefname{question}{Question}{Questions}
 \Crefname{question}{Question}{Questions}

\theoremstyle{plain}
\newtheorem{theorem}{Theorem}[subsection]
\newtheorem{lemma}[theorem]{Lemma}
\newtheorem{lem}[theorem]{Lemma}

\newtheorem{prop}[theorem]{Proposition}

\newtheorem{cor}[theorem]{Corollary}
\newtheorem{thm}[theorem]{Theorem}
\newtheorem*{thm*}{Theorem}
\newtheorem*{prop*}{Proposition}

\theoremstyle{definition}

\theoremstyle{remark}
\newtheorem{remark}[theorem]{Remark}
\newtheorem{defn}[theorem]{Definition}
\newtheorem{notn}[theorem]{Notation}
\newtheorem{conj}[theorem]{Conjecture}

\newcommand{\bfB}{ B}
\newcommand{\bfG}{G}
\newcommand{\bfH}{H}
\newcommand{\bfP}{P}
\newcommand{\bfS}{S}
\newcommand{\bfU}{U}
\newcommand{\bfX}{X}

\newcommand{\ha}{\mathfrak h}

\newcommand{\NC}{\mathcal N}

\newcommand{\N}{\mathbb N}
\newcommand{\R}{\mathbb R}
\newcommand{\C}{\mathbb C}
\newcommand{\BB}{\mathcal B}
\newcommand{\bA}{\mathbb A}
\newcommand{\bP}{\mathbb P}

\newcommand{\cE}{\mathcal E}

\newcommand{\cL}{\mathcal L}
\newcommand{\cM}{\mathcal{M}}
\newcommand{\cN}{\mathcal{N}}

\newcommand{\cO}{\mathcal O}

\newcommand{\cU}{\mathcal U}

\newcommand{\fb}{\mathfrak b}
\newcommand{\fg}{\mathfrak g}
\newcommand{\fh}{\mathfrak h}

\newcommand{\fz}{\mathfrak z}

\newcommand{\g}{\mathfrak g}
\newcommand{\h}{\mathfrak h}

\newcommand{\oH}{\operatorname H}

\newcommand{\Sc}{\mathcal{S}}
\newcommand{\Cc}{\mathrm{C}_c^{\infty}}
\newcommand{\dist}{\mathcal{D}'}
\newcommand{\Fre}{{Fr\'{e}chet \,}}
\newcommand{\alp}{{\alpha}}

\DeclareMathOperator{\Ad}{Ad}

\DeclareMathOperator{\rk}{rk}

\DeclareMathOperator{\Hom}{Hom}
\DeclareMathOperator{\End}{End}
\DeclareMathOperator{\Supp}{Supp}
\DeclareMathOperator{\Sp}{Sp}
\DeclareMathOperator{\Gr}{Gr}
\DeclareMathOperator{\Irr}{Irr}
\DeclareMathOperator{\SIS}{SS}

\newcommand{\Rami}[1]{{{#1}}}
\newcommand{\Dima}[1]{{{#1}}}
\newcommand{\DimaB}[1]{{{#1}}}
\newcommand{\DimaC}[1]{{{#1}}}
\newcommand{\DimaD}[1]{{{#1}}}
\newcommand{\DimaE}[1]{{{#1}}}
\newcommand{\DimaF}[1]{{{#1}}}
\newcommand{\DimaG}[1]{{{#1}}}
\newcommand{\DimaH}[1]{{{#1}}}
\newcommand{\Andrey}[1]{{{#1}}}

\begin{document}

\author{Avraham Aizenbud}
\address{
Faculty of Mathematics and Computer Science,
Weizmann Institute of Science,
 234 Herzl Street, Rehovot 7610001, Israel}
\email{aizenr@gmail.com}
\urladdr{\url{http://www.wisdom.weizmann.ac.il/~aizenr}}

\author{Dmitry Gourevitch}
\email{dmitry.gourevitch@weizmann.ac.il}
\urladdr{\url{http://www.wisdom.weizmann.ac.il/~dimagur}}

\author{Andrey Minchenko}

\thanks{The three authors were partially supported by the Minerva foundation with funding from the Federal German Ministry for Education and Research;
A.A. was also partially supported by ISF grant 687/13, and
\Rami{D.G. and A.M. by
\DimaF{ERC StG grant 637912 and ISF grant 756/12. At the beginning of this project D.G. was a team member in the ERC grant 291612.}}}

\email{an.minchenko@gmail.com}
\urladdr{\url{http://www.wisdom.weizmann.ac.il/~andreym/}}


\date{\today}
\title[Holonomicity of relative characters]{Holonomicity of \DimaD{relative} characters and applications to multiplicity bounds \DimaD{for spherical pairs}}

\keywords{Real reductive group, Holonomic D-module, equivariant distribution, Bessel function, spherical pair, moment map.  2010 MS Classification: 17B08,22E45,22E46,22E50,14M27, 46F05, 32C38,53D20.
}

%
%
%
%
%
%

\begin{abstract}
In this paper, we prove that any relative character (a.k.a. spherical character) of any admissible representation of a real reductive \DimaG{group}  with respect to any pair of  spherical subgroups is a holonomic distribution on \DimaG{the group}. This implies that the restriction of the relative character to an open dense subset is given by an analytic function. The proof is based on an argument from algebraic geometry and thus implies also analogous results in the p-adic case.

As an application, we give a short proof of some results from \cite{KO,KS} on boundedness and finiteness of multiplicities of irreducible representations  in the space of functions on a spherical space.

In order to deduce this application we prove relative and quantitative  analogs of the Bernstein-Kashiwara theorem, which states that the space of solutions of a holonomic system of differential equations in the space of distributions is finite-dimensional. We also deduce that, for every algebraic group $\DimaG{\bfG}$ \DimaG{defined over $\R$}, the space of $\DimaG{\bfG(\R)}$-equivariant distributions on \DimaG{the manifold of real points of} any algebraic $\DimaG{\bfG}$-manifold $\DimaG{\bfX}$ is finite-dimensional if $\DimaG{\bfG}$ has finitely many orbits on $\DimaG{\bfX}$.
\end{abstract}

\dedicatory{\DimaF{to Joseph Bernstein - the teacher, the mathematician, and the person.}}

\maketitle

\tableofcontents
\section{Introduction}
\subsection{The relative character}\( \)
\DimaG{Let $\bfG$ be a  reductive group\footnote{\DimaH{By a reductive group we mean a connected algebraic reductive group}} defined over $\R$}.
In this paper, we prove that a \DimaD{relative character (a.k.a. spherical character)} of a smooth admissible \Fre  representation of moderate growth of \DimaG{$\bfG(\R)$} is holonomic.
The \DimaD{relative} character is a basic notion of relative representation theory that generalizes the notion of a character of a representation.  Let us now recall the notions of spherical pair, \DimaD{relative} character and holonomic distribution. \Rami{For the notion of smooth admissible \Fre  representation of moderate growth \DimaF{and its relation to Harish-Chandra modules} we refer the reader to \cite{CasGlob} or \cite[Chapter 11]{Wal}}.
\begin{defn}
\DimaG{Let $\bfH\subset \bfG$ be a (closed) algebraic subgroup defined over $\R$}. Let ${P}$ denote a minimal parabolic subgroup of \DimaG{$\bfG$ defined over $\R$} and $\bfB$ denote a Borel subgroup of $\bfG$  \DimaG{(possibly not defined over $\R$}).
The subgroup $\bfH$ is called  \emph{ spherical} if it has finitely many orbits on $\bfG/\bfB$. \DimaG{We will call $\bfH$ \emph{strongly real spherical} if it has finitely many orbits on $\bfG/\bfP$}.
\end{defn}

It is known that a pair \DimaG{$(\bfG,\bfH)$ is spherical if and only if $\bfH$ has an open orbit on $\bfG/\bfB$.}


\begin{defn}\label{def:SphChar}
Let $\bfH_1,\bfH_2\subset \bfG$ be spherical subgroups \DimaC{and let $\fh_i$ be the Lie algebras of $\bfH_i$}. Let $\chi_i$ be characters of $\DimaC{\fh_i}$. Let $\pi$ be a smooth admissible \Fre representation of moderate growth of $\bfG\DimaG{(\R)}$\DimaF{, $\pi^*$ be the continuous dual of $\pi$, and $\hat{\pi}\subset \pi^*$ be the smooth contragredient representation to $\pi$ (i.e. the only smooth admissible \Fre representation of moderate growth with the same space underlying Harish-Chandra module as $\pi^*$). }
Let
$\phi_1 \in (\pi^*)^{\DimaC{\fh}_1,\chi_1}$ and $\phi_2 \in (\hat \pi^*)^{\DimaC{\fh}_2,\chi_2}$ be equivariant functionals. Fix a Haar measure on $\bfG\DimaG{(\R})$. It gives rise to an action of the space of Schwartz functions $\Sc(\bfG\DimaG{(\R}))$ on $\pi^*$ and $\hat \pi^*$, and this action maps elements of $\pi^*$ and $\hat \pi^*$ to elements of $\hat \pi$ and $\hat{\hat \pi}=\pi$ \Rami{respectively. For the definition of the space of Schwartz functions $\Sc(\bfG\DimaG{(\R}))$ see, e.g., \cite{CasGlob,Wal,AGSc}}.

The \DimaD{relative} character $\xi_{\phi_1,\phi_2}$ of $\pi$, with respect to $\phi_1$ and $\phi_2$,  is the tempered distribution on $\bfG\DimaG{(\R})$ (i.e. a continuous functional on $\Sc(\bfG\DimaG{(\R}))$) defined by $\langle \xi_{\phi_1,\phi_2} , f \rangle= \langle \phi_1, \pi(f) \cdot \phi_2 \rangle$.
\end{defn}

\begin{defn}\label{def:SISDist}
\DimaG{Let $\bfX$ be an algebraic manifold defined over $\R$. Let $\xi\in \Sc^*(\bfX(\R))$ be a tempered distribution.
The singular support\footnote{a.k.a. characteristic variety}  $\SIS(\xi)$ of $\xi$}
is the \DimaG{ zero locus} in $T^*\bfX$ of all the symbols of (algebraic) differential operators that annihilate $\xi$. The distribution $\xi$ is called holonomic if $\dim \SIS(\xi) = \dim \bfX$.
\end{defn}

In this paper we prove the following theorem.
\begin{introthm}[See \S \ref{subsec:PfMain}]\label{thm:main}
In the situation of \Cref{def:SphChar}, the \DimaD{relative} character  $\xi_{\phi_1,\phi_2}$ is holonomic.
\end{introthm}

We prove \Cref{thm:main} using the following well-known statement.

\begin{prop}[See \S \ref{subsec:PfMain}]\label{prop:S}
Let $\g,\h_i$ be the Lie algebras of $\bfG$ and $\bfH_i$, $i=1,2$.
\DimaC{Identify $T^*\bfG$ with $\bfG\times\g^*$ and let}
 $$\bfS:=\{(g,\alpha) \in \bfG \times \g^*\, \mid\, \alpha \text{ is nilpotent, } \langle \alpha, \h_1 \rangle =0, \ \langle \alpha, Ad^*(g)(\h_2) \rangle =0\}.$$

Then $\SIS(\xi_{\phi_1,\phi_2}) \subset \bfS.$
\end{prop}


Note that the Bernstein inequality states that the dimension of the singular support of any non-zero distribution is at least the dimension of the underlying manifold. Thus  Theorem \ref{thm:main} follows from the following more precise version, which is the core of this paper.

\begin{introthm}[See \S \ref{sec:geo}]\label{thm:geo}
We have $\dim \bfS = \dim \bfG$.
\end{introthm}

\DimaC{
Let $\bfU:=\left \{ g \in \bfG \, \mid \, \bfS\cap T_g^*\bfG=\{(g,0)\}\right\}.$ Note that $\bfU$ is Zariski open since $\bfS$ is conic \DimaD{and closed}. It is easy to see that \Cref{thm:geo}  implies the following corollary.
\begin{introcor}\label{cor:geo}
 The set $\bfU$ is a Zariski open dense subset of $\bfG$.
\end{introcor}
}
This corollary is useful in view of the next proposition, which follows from \Cref{prop:S,cor:SmoothDist} below.
\begin{prop}\label{prop:ResSmooth}
The restriction $\xi_{\phi_1,\phi_2}|_{\DimaG{\bfU(\R)}}$ is an analytic function.
\end{prop}

\DimaF{
\begin{remark}
In general, $\bfS$ has irreducible components that can not lie in
$\SIS(\xi_{\phi_1,\phi_2})$ for any $\phi_1,\phi_2$. Indeed, $\SIS(\xi_{\phi_1,\phi_2})$ is coisotropic by \cite{Gab,KKS,Mal}, and thus, by Theorem \ref{thm:main}, Lagrangian. On the other hand, one can show that when \DimaG{$\bfG=\mathrm{GL}_{4,\R}$, and $\bfH_1=\bfH_2=\mathrm{GL}_{2,\R}\times \mathrm{GL}_{2,\R}$} embedded as block matrices inside $\bfG$, the variety $\bfS$ has non-isotropic (and thus non-Lagrangian) components.
\end{remark}
}

\subsection{Bounds on the dimension of the space of solutions}

Next we apply our results to representation theory. For this, we use the following theorem.

\begin{thm}[Bernstein-Kashiwara]\label{thm:DimSol}
Let \DimaG{$\bfX$ be \Rami{an algebraic manifold} defined over $\R$}. Let $$\{D_i \xi=0\}_{i=1... n}$$ be a system of linear PDE on $\DimaG{\Sc^*(\bfX(\R))}$ with algebraic coefficients. Suppose that the joint zero set of the symbols of $D_i$ in $T^*\bfX$ is $\dim \bfX$-dimensional. Then the space of solutions of this system  is finite-dimensional.
\end{thm}

It seems that this  theorem is not found in the literature in this formulation, however it has two proofs, one due to Kashiwara (see \cite{Kas,KK} for similar statements) and another due to Bernstein (unpublished).

In order to make our applications in representation theory more precise, we need an effective version of  this theorem. We prove such a version (see Theorem \ref{thm:FinSol} below) following Bernstein's approach, as it is more appropriate for effective bounds. We use this effective version to derive a relative version. Namely, we show that if the system depends on a parameter in an algebraic way, then the dimension of the space of solutions is bounded (see \S \ref{subsec:DimSolRel}  below).

This relative version allows us to deduce the following theorem.
\DimaD{
\begin{introthm}[See \S \ref{subsec:DimSolRel}]\label{thm:FinOrb}
Let \DimaG{$\bfG$ be an  algebraic group defined over $\R$ and let $\bfX$ be an  algebraic $\bfG$-manifold} with finitely many orbits. Let $\fg$ be the Lie algebra of $\bfG$. Let $\cE$ be an algebraic $\bfG$-equivariant bundle on  $\bfX$.  Then, for any natural number $n\in \N$, there exists $C_n \in  \mathbb{N}$ such that for every $n$-dimensional representation $\tau$ of $\fg$ we have  $$\dim \Hom_{\fg}(\tau, \Sc^*(\bfX(\R),\cE))\leq C_{n},$$
where $\Hom_{\fg}$ denotes the space of all continuous $\fg$-equivariant maps.

\end{introthm}

\DimaG{
\begin{remark}
Note that the condition that $\bfG$ has finitely many orbits on $\bfX$ is equivalent to $\bfG(\C)$ having finitely many orbits on $\bfX(\C)$ but {\bf not} equivalent  to (and not implied by) $\bfG(\R)$ having finitely many orbits on $\bfX(\R)$.
\end{remark}
}

\subsection{Applications to representation theory}

Using \S \ref{sec:DimSol}, we give  a short proof of some results \DimaE{from} \cite{KO,KS}. Namely, we prove:
\begin{introthm}[See \S \ref{sec:Rep}]\label{thm:IntroMult} Let \DimaG{${G}$ be a reductive group defined over $\R$, $\bfH\subset \bfG$} be a Zariski closed subgroup, and $\fh$ be the Lie algebra of $\bfH$.
\begin{enumerate}[(i)]
\item \label{it:Fin} If \DimaG{$\bfH$ is a strongly} real spherical subgroup then, {for every \DimaE{irreducible smooth admissible \Fre  representation of moderate growth $\pi\in \Irr(\DimaG{\bfG(\R))}$}, and natural number $n\in \N$ there exists $C_n \in  \mathbb{N}$ such that} for every $n$-dimensional representation $\tau$ of $\fh$ we have  $$\dim \Hom_{\fh}(\pi,\tau) \leq C_n.$$
\item \label{it:Bound} If $H$ is a spherical subgroup and we consider only one-dimensional $\tau$ then the space is universally bounded, {\it i.e.}
there exists $C \in  \mathbb{N}$ such that $\dim (\pi^*)^{\fh,\chi} \leq C$ for any $\pi\in \Irr(\DimaG{\bfG(\R))})$ \Andrey{and any character $\chi$ of $\fh$}.
\end{enumerate}
\end{introthm}

\begin{introcor}\label{cor:strong}

 Let \DimaG{${G}$ be a reductive group defined over $\R$, $H\subset G$} be a Zariski closed reductive subgroup, and $\fh$ be the Lie algebra of $H$.

\begin{enumerate}[(i)]
\item \label{it:StFin} If the diagonal $\Delta \bfH$ is a \DimaG{strongly} real spherical subgroup in $\bfG\times \bfH$ then for every $\pi\in \Irr(\DimaG{\bfG(\R))})$ and $\tau \in \Irr(\DimaG{\bfH(\R))})$ we have finite multiplicities, i.e. $$\dim \Hom_{\fh}(\pi,\tau)< \infty .$$
\item \label{it:StBound}  If the diagonal $\Delta\bfH$ is a spherical subgroup in $\bfG\times \bfH$ then the multiplicities are universally bounded, i.e.,
there exists $C \in  \mathbb{N}$ such that for every $\pi\in \Irr(\DimaG{\bfG(\R))})$, $\tau \in \Irr(\DimaG{\bfH(\R))})$ we have $$\dim \Hom_{\fh}(\pi,\tau) \leq C .$$
\end{enumerate}
\end{introcor}
This corollary follows from Theorem \ref{thm:IntroMult} since $\Hom_{\fh}(\pi,\tau)$ lies in the space of $\Delta \fh$-invariant functionals on the completed tensor product $\pi \widehat{\otimes} \hat \tau\in \Irr(\DimaG{\bfG(\R))}\times \DimaG{\bfH(\R))})$ (see \cite[Corollary A.0.7 and Lemma A.0.8]{AMOT}). All symmetric pairs satisfying the conditions of the corollary were classified in \cite{KM}.

The inverse implications for Theorem \ref{thm:IntroMult}\DimaG{\eqref{it:Bound}} and Corollary \ref{cor:strong}\DimaG{\eqref{it:StBound}} are proven in \cite{KO}.



The results on multiplicities in \cite{KO,KS} are  stronger than Theorem \ref{thm:IntroMult} since they \DimaG{do not require  $H$ to be algebraic,}  and consider maps from the Harish-Chandra space of $\pi$ to $\tau$.
\DimaG{Also, in \cite{KO,KS} Theorem \ref{thm:IntroMult}\eqref{it:Fin} and Corollary \ref{cor:strong}\eqref{it:StFin} are proven in the wider generality of real spherical subgroups.}

In addition,
\cite[Theorem B]{KO} implies that if  $H\subset G$ is an algebraic spherical subgroup there exists  $C \in  \mathbb{N}$ such that $\dim \Hom_{\fh}(\pi,\tau) \leq C\dim \tau,$ for every $\pi\in \Irr(G\DimaG{(\R)})$ and every finite-dimensional continuous representation $\tau$ of $H\DimaG{(\R)}$. It is easy to modify our proof of Theorem \ref{thm:IntroMult}\eqref{it:Bound} to show the boundedness of multiplicities for any $\pi\in \Irr(G\DimaG{(\R)})$  and any $\tau$ of a fixed dimension, but the proof that the bound depends linearly on this dimension would require more work.

Our methods are different from the methods of \cite{KO}, which in turn differ from the ones of \cite{KS}, and the bounds given in the three works are probably very different.
}
\subsection{The non-Archimedean case}
 Theorem \ref{thm:geo} and Corollary \ref{cor:geo} hold over arbitrary fields of characteristic zero. They are useful also for $p$-adic local fields $F$, since the analogs of \Cref{prop:S,prop:ResSmooth} hold in this case, see \cite[Theorem A and Corollary F]{AGS}.
Namely, we have the following theorem.
\begin{thm}[\cite{AGS}]\label{thm:AGS}
Let $G$ be a reductive group defined over a non-Archimedean field $F$ of characteristic $0$ and let $\xi$ be a \DimaD{relative} character of a smooth admissible  representation with respect to two spherical subgroups $H_1,H_2\subset G$.
Let $S$ and $U$ be the sets defined in \Cref{prop:S,cor:geo}. 
%
Then
\begin{enumerate}[(i)]
\item \label{thm:AGS:S} The wave front set of $\xi$ lies in $S\DimaG{(F)}$.

\item\label{thm:AGS:U} The restriction of $\xi$ to $U\DimaG{(F)}$  is given by a locally constant function.

\end{enumerate}
\end{thm}
%

\subsection{Related results}
In the group case, i.e. the case when $G=H\times H$ and $H_1=H_2=\Delta H \subset H \times H$, Theorem \ref{thm:main} essentially becomes the well-known fact that characters of admissible representations are holonomic distributions.

As we mentioned above, Theorem \ref{thm:IntroMult} was proven earlier in \cite{KO,KS} using different methods. An analog of Theorem \ref{thm:IntroMult}\eqref{it:Fin} over non-Archimedean fields is proven in \cite{Del} and \cite[Theorem 5.1.5]{SV} for many spherical pairs, including arbitrary symmetric pairs.

The group case of \Cref{cor:geo}, \Cref{prop:ResSmooth}, and \Cref{thm:AGS}\eqref{thm:AGS:U} is (the easy part of) the Harish-Chandra regularity theorem (see \cite{HCBul,HCReg}).  \Dima{Another known special case of these results is the regularity of Bessel functions, see \cite{LM,AGK,AG}.}

\subsection{Future applications}
Our proof of Theorem \ref{thm:IntroMult}\eqref{it:Bound} does not use the Casselman embedding theorem {(Theorem~\ref{thm:CasSubRep})}. 
This gives us hope that it can be extended to the non-Archimedean case. The main difficulty is the fact that our proof heavily relies on the theory of modules over the ring of differential operators, which does not act  on distributions in the non-Archimedean case. However, in view of  \Cref{thm:AGS} we believe that this difficulty can be overcome. Namely, one can deduce an analog of Theorem \ref{thm:IntroMult}\eqref{it:Bound}  for many  spherical pairs from the following conjecture .
\begin{conj}\label{conj:fin}
Let $G$ be a reductive group defined over a non-Archimedean field $F$ of characteristic $0$ and let $H_1,H_2\subset G$ be its (algebraic) spherical subgroups. Let $\chi_i$ be characters of $H_i\DimaG{(F)}$.
Fix a character $\lambda$ of the Bernstein center $\fz(G\DimaG{(F)})$.

Then the space of distributions which are:

\begin{enumerate}
\item left $(H_1\DimaG{(F)},\chi_1)$-equivariant,
\item right $(H_2\DimaG{(F)},\chi_2)$-equivariant,
\item $(\fz(G\DimaG{(F)}),\lambda)$-eigen,
\end{enumerate}
is finite-dimensional. Moreover, this dimension is uniformly bounded when $\lambda$ varies.
\end{conj}
Note that \Cref{thm:geo} and \Cref{thm:AGS}\eqref{thm:AGS:S} imply that the dimension of (the Zariski closure of) the wave front set of a distribution that satisfies (1-3)  does not exceed $\dim G$.  In many ways the wave front set replaces the singular support, in absence of the theory of differential operators (see, e.g., \cite{Aiz,AD,AGS,AGK}). Thus, in order to prove \Cref{conj:fin}, it is left  to prove analogs of   \Cref{thm:DimSol,thm:FinSol} for the integral system of equations (1-3).

\subsection{Structure of the paper}
In \S \ref{sec:geo}, we prove Theorem \ref{thm:geo} using \Rami{a theorem of Steinberg \cite{Ste76} concerning the Springer resolution.}

In \S \ref{sec:DimSol}, we prove an effective version of Theorem \ref{thm:DimSol}, and then adapt it to algebraic families. We also derive Theorem \ref{thm:FinOrb}.

In \S \ref{sec:Rep}, we derive Theorem \ref{thm:IntroMult}  from Theorem \ref{thm:geo} and \S \ref{sec:DimSol}. We do that by embedding the multiplicity space into a certain space of \DimaD{relative} characters.

In  Appendix \ref{app:dist}, we prove \Cref{lem:PullDist} which computes the pullback of the D-module of distributions with respect to a closed embedding. We use this lemma in \S \ref{sec:DimSol}.
%

\subsection{Acknowledgements}
We thank Eitan Sayag and Bernhard Kroetz for fruitful discussions.
We thank Joseph Bernstein for telling us the sketch of his proof of Theorem \ref{thm:DimSol}.
\DimaG{We thank Toshiyuki Kobayashi, Alexander Shamov, and the anonymous referee for useful remarks.}

\DimaF{A.A. and D.G. will always be grateful to Joseph Bernstein for introducing them to the amazing world of algebra, for sharing his knowledge, his approaches to problems and his philosophy for already more than half of their lives, and for being a shining example forever.}

\Rami{
\section{Proof of Theorem \ref{thm:geo}}\label{sec:geo}
It is enough to prove the theorem for a reductive group $G$ defined over an algebraically closed field 
of characteristic $0$.
Since $S$ includes the zero section of $\DimaB{T^*G \cong }G \times \fg^*$, we have $\dim S \geq \dim G$. Thus, it is enough to prove that $\dim S \leq \dim G$.
Let $\BB$ denote the flag variety of $G$ and $\cN\subset \fg^*$ denote the nilpotent cone.
Since $G$ is reductive, we can identify $$T^*\BB \cong \{ (B,X) \Dima{\in \BB \times \fg^*} \, \vert \, X \in (\mathrm{Lie} B)^{\bot} \}.$$  Recall the Springer resolution $\mu:T^*\BB \to \cN$ defined by $\mu(B,X)=X$ and consider the following diagram.

\begin{equation}\label{eq:smalldiag}
\xymatrixcolsep{5pc}
\xymatrix{
 T^*\BB \times T^*\BB \ar@{->>}[rd]^{\mu \times \mu}  &  & G\times \cN\ar[ld]^{\alpha} \\
  & \cN \times \cN  \ar[d]^{res} & \\
  & \fh_1^* \times \fh_2^* &
}
\end{equation}


Here, $\alpha$ is defined by $\alpha(g,X)=(X,\Ad^*(g^{-1})X)$, and $res$ is the restriction.
Passing to the fiber of $0\in \fh_1^* \times \fh_2^*$, we obtain the following diagram.

\begin{equation}\label{eq:fiberdiag}
\xymatrixcolsep{5pc}
\xymatrix{
 L_1 \times L_2 \ar@{->>}[rd]^{\mu'}  &  & S\ar[ld]^{\alpha'}\\
  & \cN_{\fh_1} \times \cN_{\fh_2}  &
 }
\end{equation}


Here, ${\NC}_{\fh_i}:=\NC \cap \fh_i^{\bot}$ and $L_i:=\{ (B,X) \in T^*\BB \, \vert \, X \in \fh_i^{\bot} \}$.
We need to estimate $\dim S$.
We do it using the following lemma.
\begin{lemma}[See \S \ref{sec:lemma} below]\label{lem:tech}
Let $\varphi_i:X_i\to Y$, $i=1,2$, be morphisms of algebraic varieties. Suppose that $\varphi_2$ is surjective. Then there exists $y\in Y$ such that
$$
\dim X_1\leq \dim X_2 + \dim \varphi_1^{-1}(y) - \dim\varphi_2^{-1}(y).
$$
\end{lemma}

By this lemma, applied to $\phi_1=\alp'$ and $\phi_2=\mu'$, it is enough to estimate the dimensions of $L_i$ and of the fibers of $\mu'$ and $\alpha'$.
\begin{lemma}
We have $\dim L_1=\dim L_2 = \dim \BB.$
\end{lemma}
\begin{proof}
Since $H_i$ has finitely many orbits in $\BB$, it is enough  to show that $L_i$ is the union of the conormal bundles to the orbits of $H_i$ in $\BB$.
Let $B\in \BB$, and $\fb=\mathrm{Lie} B$, and identify $T_B\BB \cong \fg/\fb$. Then $T_B(H_i\cdot B)\cong \fh_i /(\fb\cap \fh_i)$ and the conormal space at $B$ to the $H_i$-orbit of $B$ is identified with $ \fb^\bot \cap \fh_i^\bot$.
\end{proof}

Let $(\eta, \Ad^*(g)\eta)\in \mathrm{Im}(\alpha')$. The fiber $(\alpha')^{-1}(\eta, \Ad^*(g)\eta)$ is isomorphic to the stabilizer $G_{\eta}$, and the dimension of the fiber $(\mu')^{-1}(\eta, \Ad^*(g)\eta)$ is twice the dimension of the Springer fiber $\mu^{-1}(\eta)$.
Recall the following  theorem of Steinberg (conjectured by Grothendieck):
\begin{theorem}[{\cite[Theorem 4.6]{Ste76}}]\label{thm:GrConj}
$$
\dim G_\eta-2\dim\mu^{-1}(\eta)=\rk G.
$$
\end{theorem}
Using Lemma \ref{lem:tech}, we obtain for some $(\eta, ad^*(g)\eta)$:
\begin{multline*}
\dim S \leq \dim (L_1\times L_2) + \dim (a')^{-1}(\eta, ad^*(g)\eta) - \dim (\mu')^{-1}(\eta, ad^*(g)\eta) =\\= 2\dim \BB +  \dim G_\eta-2\dim\mu^{-1}(\eta)=2\dim \BB + \rk G = \dim G.
\end{multline*}
\qed

\subsection{Proof of Lemma \ref{lem:tech}}\label{sec:lemma}
Recall that, for a dominant morphism $\varphi: X\to Y$ of irreducible varieties, there exists an open dense $U\subset Y$ such that $\dim X=\dim Y+\dim\varphi^{-1}(y)$ for all $y\in U$ (see, e.~g.,~\cite[Theorem 1.8.3]{RedBook}).
Let $Z$ be an irreducible component of $X_1$ \Dima{of maximal dimension} and $W\subset Y$ be the Zariski closure of $\varphi_1(Z)$. Since $W$ is irreducible, there exists an open dense $U\subset W$ such that
\begin{equation}\label{eq:dim1}
\Dima{\dim X_1 =} \dim Z\leq \dim W+\dim\varphi_1^{-1}(y)
\end{equation}
for all $y\in U$.
Let $V\subset U$ be an open dense subset such that $\varphi_2^{-1}(V)$ intersects those and only those irreducible components $C_1,\ldots,C_j$ of $\varphi_2^{-1}(W)$ that map dominantly to $W$. \Dima{Note that $j>0$ since $\varphi_2$ is surjective.} 
Moreover, without loss of generality, we may assume that, for every $1\leq i\leq j$, all fibers over $V$ of the restriction of $\varphi_2$ to $C_i$ are of the same dimension. Since one of these dimensions has to be equal to $\dim\varphi_2^{-1}(V)-\dim V$, we have, that there is an $1\leq i\leq j$ such that, for all $y\in V$,
\begin{equation}\label{eq:dim2}
\dim V=\dim C_i - \dim(\varphi_2|_{C_i})^{-1}(y) \leq \dim\varphi_2^{-1}(V)-\dim\varphi_2^{-1}(y)\leq \dim X_2-\dim\varphi_2^{-1}(y).
\end{equation}
Thanks to $\dim V=\dim W$, taking any $y\in V$, formulas~\eqref{eq:dim1} and~\eqref{eq:dim2} imply the statement. \qed
}
\section{Dimension of the space of solutions of a holonomic system}\label{sec:DimSol}
In this section, we prove an effective version of Theorem \ref{thm:DimSol}, and then adapt it to algebraic families. We also derive Theorem \ref{thm:FinOrb}.


\subsection{Preliminaries}\label{subsec:DimSolPrel}

\subsubsection{D-modules}
In this section, we will use the theory of D-modules \Rami{on algebraic varieties over an arbitrary field $k$ of characteristic zero}.
We will now recall some facts and notions that we will use.
For a good introduction to the algebraic theory of D-modules, we refer the reader to \cite{Ber} and \cite{Bor}. For a short overview, see \cite[Appendix B]{AMOT}.

By a \emph{D-module} on a smooth algebraic variety $X$ we mean a \DimaF{quasi-}coherent sheaf of right modules over the sheaf $D_X$ of algebras of algebraic differential operators. \DimaF{ By a \emph{finitely generated D-module} on a smooth algebraic variety $X$ we mean a coherent sheaf of right modules over the sheaf $D_X$.}
Denote the category of $D_X$-modules by $\cM(D_X)$.

For a smooth affine variety $V$, we denote $D(V):=D_V(V)$.
Note that the category $\cM(D_V)$ of D-modules on $V$ is equivalent to the category of $D(V)$-modules.
We will thus identify these categories.

The algebra $D(V)$ is equipped with a filtration which is called the geometric filtration and defined by the degree of differential operators. The associated graded algebra with respect to this fitration is the algebra $\cO(T^*V)$ of regular functions on the total space of the cotangent bundle of $V$. This allows us to define the \emph{singular support} of a \DimaF{finitely generated} D-module $M$ on $V$ in the following way. \Rami{Choose a good filtration on $M$, i.e. a filtration such that the associated graded module is a finitely-generated module over $\cO(T^*V)$, and define the singular support $SS(M)$ to be the support of this module. One can show that the singular support does not depend on the choice of a good filtration on $M$.}

This definition easily extends to the non-affine case. A \DimaF{finitely generated} D-module $M$ on $X$ is called \emph{smooth} if $SS(M)$ is the zero section of $T^*X$. This is equivalent to being coherent over $\cO_X$ and to being coherent and locally free over $\cO_X$. The Bernstein inequality states that, for any non-zero \DimaF{finitely generated} $M$, we have $\dim SS(M) \geq \dim X$. If the equality holds then $M$ is called \emph{holonomic}.

For a closed embedding $i:X \to Y$ of smooth affine algebraic varieties, define the functor $i^!:\cM(D_Y)\to \cM(D_X)$ by $i^!(M):=\{m\in M \, \vert \, I_Xm=0\},$
where $I_X\subset \cO(Y)$ is the ideal of all functions that vanish on $X$.
\DimaC{It has a left adjoint functor $i_*:\cM(D_X) \to \cM(D_Y)$, given by tensor product with $i^!(D_Y)$. The functor $i_*$ is an equivalence of categories between $\cM(D_X)$ and the category of $D_Y$-modules supported in $X$. Both $i_*$ and $i^!$ map holonomic modules to holonomic ones.}

If $V$ is an affine space then the algebra $D(V)$ has an additional filtration, called the Bernstein filtration. It is defined by $\deg(\partial/\partial x_i)=\deg(x_i)=1,$  where $x_i$ are the coordinates in $V$. This gives rise to the notion of Bernstein's singular support, that we will denote $SS_b(M)\subset T^*V\cong V\oplus V^*$. It is known that $\dim SS(M)=\dim SS_b(M)$.

We will also use the theory of analytic D-modules. By an \emph{analytic D-module} on a smooth \Rami{complex} analytic manifold $X$ we mean a coherent sheaf of right modules over the sheaf $D_X^{An}$ of algebras of  differential operators with analytic coefficients. 
All of the above notions and statements, except those concerning the Bernstein filtration, have analytic  counterparts. In addition, all smooth analytic D-modules of the same rank are isomorphic.

\subsubsection{Distributions}
We will use the theory of distributions on differentiable manifolds and the theory of tempered distributions on real algebraic manifolds, see e.g. \cite{Hor,AGSc}.
For \DimaG{an  algebraic manifold $X$ defined over $\R$, we denote the space of \DimaG{real valued distributions on $X(\R)$ by $\dist(X(\R),\R):=(C_c^{\infty}(X(\R),\R))^*$} and the space of real valued tempered  distributions (a.k.a. Schwartz distributions) by $\Sc^*(X\DimaG{(\R)},\R):=(\Sc(X\DimaG{(\R),\R}))^*$. Similarly, denote by $\dist(X(\R)):=\dist(X(\R),\C)$ the space of complex valued distributions on $X(\R)$ and by $\Sc^*(X(\R)):=\Sc^*(X(\R),\C)$ the space of complex valued tempered distributions on $X(\R)$. Also,
for an algebraic bundle $\cE$ (complex or real) over $X$ we denote $\dist(X\DimaG{(\R)},\cE):=(C_c^{\infty}(X\DimaG{(\R)},\cE))^* $ and  $\Sc^*(X\DimaG{(\R)},\cE):=(\Sc(X\DimaG{(\R)},\cE))^*$.}


\DimaG{
The algebra $D(X)$  acts on the spaces $\dist(X\DimaG{(\R),\R})$ and $\Sc^*(X\DimaG{(\R),\R})$.
Thus, for affine $X$ we can consider these spaces as $D_{X}$-modules and $\dist(X\DimaG{(\R),\R})$  also as an analytic $D_X$-module.
Similarly, we will regard the spaces $\dist(X(\R))$ and $\Sc^*(X(\R))$ as $D_{X_{\C}}$-modules, where $X_{\C}$ denotes the complexification of $X$.
For non-affine $X$, define the $D_X$-module $\Sc^*_{X,\R}$ by $\Sc^*_{X,\R}(U):=\Sc^*(U(\R),\R)$.
It is easy to see that this sheaf is quasi-coherent. Denote by $\Sc^*_{X}$ the
$D_{X_{\C}}$-module obtained by
complexification of $\Sc^*_{X,\R}$. } 

We define the singular support of a \DimaG{tempered} distribution to be the singular support of the D-module it generates. \Rami{It is well known  that this definition \DimaB{is equivalent to} Definition \ref{def:SISDist}.}
We say that a distribution is holonomic if it generates a holonomic D-module.

\begin{lem}[See \Cref{app:dist}]\label{lem:PullDist}
Let $i:X \to Y$ be a closed embedding of smooth affine \DimaG{ algebraic varieties defined over $\R$}.
Then
$$\dist(X\DimaG(\R))\cong i^!(\dist(Y\DimaG(\R))) \text{ and }\Sc^*(X\DimaG(\R))\cong i^!(\Sc^*(Y\DimaG(\R))).$$
\end{lem}

\begin{lem}\label{lem:Free}
Let $M$ be a smooth $\DimaG{D(\bA^n_{\C} )}$-module of rank $r$. \Rami{Embed the space $An(\C^n)$ of analytic functions on $\C^n$ into $\dist(\R^n)$ using the Lebesgue measure.} Then $$\Hom_{\DimaG{D(\bA^n_{\C} )}}(M,\dist(\R^n))=\Hom_{\DimaG{D(\bA^n_{\C} )}}(M,An(\C^n)) \text{ and }\dim \Hom_{\DimaG{D(\bA^n_{\C} )}}(M,\dist(\R^n))= r.$$
\end{lem}
\begin{proof}
Let $M_{An}:=M \otimes_{\cO(\C^n)}An(\C^n)$ and $\DimaG{D_{An}(\bA^n_{\C} )}:=\DimaG{D(\bA^n_{\C} )} \otimes_{\cO(\C^n)}An(\C^n)$   be the analytizations of $M$ and $\DimaG{D(\bA^n_{\C} )}$. Then $$\Hom_{\DimaG{D(\bA^n_{\C} )}}(M,\dist(\R^n)) \cong \Hom_{\DimaG{D_{An}(\bA^n_{\C} )}}(M_{An},\dist(\R^n)).$$
Since $M_{An}$ is also smooth, 
$M_{An} \cong An(\C^n)^r$. Thus it is left to prove that
$$\Hom_{\DimaG{D_{An}(\bA^n_{\C} )}}( An(\C^n),\dist(\R^n))=\Hom_{\DimaG{D_{An}(\bA^n_{\C} )}}( An(\C^n),An(\C^n))$$ and the latter space is one-dimensional.  This follows from the fact that a distribution with vanishing partial derivatives is a multiple of the Lebesgue measure.
\end{proof}

\begin{cor}\label{cor:SmoothDist}
If a distribution generates a smooth $D$-module, then it is analytic.
\end{cor}

\Rami{
\subsubsection{Lie algebra actions}

\begin{defn}Let $X$ be an algebraic manifold \DimaG{defined} over a field $k$ and $\fg$ be a  Lie algebra over $k$.
\begin{enumerate}[(i)]
\item An action of $\fg$ on $X$ is a Lie algebra map from $\fg$ to the algebra of algebraic vector fields on $X$.
\item Assume that $X$ is affine, fix an action of $\fg$ on $X$ and let $\cE$  be an algebraic vector bundle on $X$. Let $M$ be the space of global \DimaB{regular} {(algebraic)} sections of $\cE$. An action of $\fg$ on $\cE$ is a linear map $T:\fg \to \End_{k}(M)$ such that, for any $\alp\in \fg , \, f\in \cO(X), \, v \in M$, we have
    $$T(\alp)(fv)=(\alp f) v+fT(\alp)v.$$
\item The definition above extends to non-affine $X$ in a straightforward way.
\end{enumerate}
\end{defn}
}
\subsubsection{Weil representation}

\begin{defn}
Let $V$ be a finite-dimensional real vector space. Let $\omega$ be  the standard symplectic form on $V \oplus V^*$. Denote by $p_V:V \oplus V^* \to V$ and $p_{V^*}:V \oplus V^*\to V^*$ the natural projections.
Define an action of the symplectic group $\DimaG{\Sp(V \oplus V^*)}$ on the algebra $D(V)$
by
$$(\partial_v)^g:=\pi(g)(\partial_v):=p_{V^*}(g(v,0)) + \partial_{p_{V}(g(v,0))}, \quad w^g:=\pi(g)w:=p_{V^*}(g(0,w)) + \partial_{p_{V}(g(0,w))}$$
where $v\in V, \, w\in V^*, \, \partial_v$ denotes the derivative in the direction of $v$, and elements of $V^*$ are viewed as linear polynomials and thus differential operators of order zero.
For a  $D(V)$-module $M$ and an element $g\in \Sp(V \oplus V^*)$, we will denote by $M^g$ the $D(V)$-module obtained by twisting the action of $D(V)$ by $\pi(g)$.
\end{defn}

Since the above action of $\Sp(V \oplus V^*)$ preserves the Bernstein filtration on $D(V)$,  the following lemma holds.
\begin{lem}\label{lem:SympSupp}
For a  \DimaF{finitely generated} $D(V)$-module $M$ and $g\in \Sp(V \oplus V^*)$ we have $SS_{b}(M^g)=gSS_{b}(M)$.
\end{lem}

\begin{thm}[\cite{Weil}]
There exists a two-folded cover $p:\widetilde{\Sp}(V \oplus V^*)\to \Sp(V \oplus V^*)\DimaG(\R)$ and a representation $\Pi$ of $\widetilde{\Sp}(V \oplus V^*)$ on the space $\Sc^*(V)$ of tempered distributions on $V$ such that, for any $\alpha \in D(V)$, $g\in \widetilde{\Sp}(V \oplus V^*)$, {$\xi\in\Sc^*(V)$,} we have$$\Pi(g)(\xi\alpha)=(\Pi(g)\xi)\alpha^{p(g)}.$$
\end{thm}

\begin{cor}\label{cor:Weil}
We have an isomorphism of  $D(V)$-modules $\Sc^*(V)^g \cong \Sc^*(V)$ for any $g\in \Sp(V \oplus V^*)\DimaG(\R)$.
\end{cor}
In fact, this corollary can be derived directly from the Stone-von-Neumann theorem.
\subsubsection{Flat morphisms}

\begin{lemma}\label{lem:flat}
Let $\phi:X \to Y$ be a proper morphism of algebraic varieties \DimaG{defined} \Rami{over a field $k$} and $\cM$ be a coherent sheaf on $X$. Then there exists an open dense $U\subset Y$ such that $\cM|_{\phi^{-1}(U)}$ is flat over $U$.
\end{lemma}
\begin{proof}
By \cite[Th{\'e}or{\`e}me II.3.I]{EGA4.3}, the set $V$ of {scheme-theoretic} points $x\in X$ for which $\cM$ is $\phi$-flat at $x$ is open in $X$. Since $\phi$ is proper, the set $Z:=\phi(X\setminus{V})$ is closed in $Y$. Note that $\cM$ is flat over $U:=X\setminus{Z}$, since $\phi^{-1}(U)\subset V$. Moreover, $U$ contains the generic points of the irreducible components of $Y$. Hence, $U\subset Y$ is dense.
\end{proof}

\begin{lemma}[{See, e.~g., \cite[Corollary on p.~50]{Mumford}}]\label{lem:loc_const}
Let $\phi:X \to Y$ be a proper morphism of algebraic varieties \DimaG{defined over a field $k$}  and $\cM$ be a coherent sheaf on $X$ that is flat over $Y$. \Rami{For a point $y\in Y$, let $\cM_y$ denote the pullback of $\cM$ to $\phi^{-1}(y)$.}
Then the function
$$y \mapsto \chi(\cM_y)=\sum_{i=0}^{\infty}(-1)^i\dim_{k(y)} \oH^i(\cM_y)$$
is locally constant.
\end{lemma}
\Rami{
\begin{cor}\label{cor:loc_const}
Let $Y$ be an algebraic variety \DimaG{defined over a field $k$} and $\cM$ be a coherent sheaf on $Y \times \bP^n$. Then there exists an open dense $U\subset Y$ such that the Hilbert polynomial\footnote{For the definition of Hilbert polynomial see \cite[Chapter III, Exercise 5.2]{Har}.} of $\cM_y$ does not depend on $y$ as long as $y\in U$.
\end{cor}
}
\subsection{Dimension of the space of solutions of a holonomic system}\label{subsec:DimSol}

\begin{defn} \DimaG{Let $k$ be a field of characteristic zero.}
\begin{enumerate}[(i)]
\item 
Let $M$ be a \DimaF{finitely generated} $D$-module over an affine space \DimaG{$\mathbb{A}_k^n$ over $k$}. Let $F^i$ be a good filtration on $M$ with respect to the Bernstein filtration on the ring $D_{\DimaG{\mathbb{A}_k^n}}$. Let $p$ be the corresponding Hilbert polynomial of $M$, i.e. $p(i)=\dim F^i$ for large enough $i$. Let $\DimaG{d}$ be the degree of $p$ and $a_{\DimaG{d}}$ be the leading coefficient of $p$. Define the Bernstein degree of $M$ to be $\deg_b(M):=\DimaG{d}!a_{\DimaG{d}}$. It is well-known that $\DimaG{d}$ and $a_{\DimaG{d}}$ do not depend on the choice of good filtration $F^i$.
\item Let $M$ be a \DimaF{finitely generated} $D$-module over a smooth algebraic variety $X$ \DimaG{defined over $k$}. Let $X=\bigcup_{i=1}^l U_i$ be an open affine cover of $X$ and let $\phi_i:U_i \hookrightarrow \DimaG{\mathbb{A}_k^{n_i}}$ be closed embeddings. Denote
$$\deg_{\{(U_i,\phi_i)\}}(M):=\sum_{i=1}^l \deg_b((\phi_i)_*(M\DimaC{|_{U_{i}}})).$$
Define the global degree of $M$ by $\deg(M):=\min  \deg_{\{(U_i,\phi_i)\}}(M)$, where the minimum is taken over the set of all possible affine covers and embeddings.
\end{enumerate}
\end{defn}

In this \Rami{subsection}, we prove
\begin{thm}\label{thm:FinSol}
Let $X$ be \DimaG{an algebraic manifold defined over $\R$}. Let $M$ be a holonomic right $D_{\DimaG{X_{\C}}}$-module. Then $\dim \Hom(M,\Sc^*(X)) \leq \deg(M)$.
\end{thm}


We will need the following geometric lemmas

\begin{lemma}\label{lem:PrFin}
Let $V$ be a vector space \DimaG{over an algebraically closed field}, $L\subset V$ be a subspace and $C\subset V$ be a closed conic algebraic subvariety such that $L\cap C = \{0\}$. Then the projection $p:C \to V/L$ is a finite map.
\end{lemma}
\begin{proof}
By induction, it is enough to prove the case $\dim L =1$. Choose coordinates $x_1,\dots,x_n$ on $V$ such that the coordinates $x_1,\dots,x_{n-1}$ vanish  on $L$. Let $p$ be a homogeneous polynomial that vanishes on $C$ but not on $L$. Write $p=\sum_{i=1}^{\DimaG{d}} g_ix_n^i$, where each $g_i$ is a homogeneous polynomial of degree $\DimaG{d}-i$ in $x_1,\dots,x_{n-1}$. Then $x_n|_C$ satisfies a monic polynomial equation with coefficients in $\cO(V/L)$.
\end{proof}

\begin{lemma}\label{lem:Lag}
Let $W$ be a $2n$-dimensional \DimaG{real} symplectic vector space, and $C\subset \DimaH{W_{\C}}$ be a closed conic   subvariety of dimension $n$. Then there exists a \DimaG{real} Lagrangian subspace $L\subset W$ such that $\DimaG{L_{\C}\cap C} = \{0\}$.
\end{lemma}
\begin{proof}
Let $\cL$ denote the variety of all Lagrangian subspaces of $W$. Note that $\dim \cL = n(n+1)/2$. Let $P(C)\subset \bP(W_{\DimaH{\C}})$ be the projectivizations of $C$ and $W$.
Consider the configuration space $$X:=\{(x,L)\in P(C) \times \cL \, \vert \, x \subset L\}.$$
\DimaG{Since $\cL$ is smooth, irreducible, and has a real point, it is enough}
 to show that $\DimaH{\overline{p(X)}} \neq \cL$ where $p:X \to \cL$ is the projection. Let $q:X\to P(C)$ be the other projection. Note that $\dim q^{-1}(x)=n(n-1)/2$ for any $x \in P(C)$. Thus $$\dim X = n(n-1)/2+n-1<n(n+1)/2=\dim \cL,$$
and thus $p:X \to \cL_{\DimaH{\C}}$ cannot be \DimaH{dominant}.
\end{proof}

\begin{cor}\label{cor:PrLag}
Let $V$ be a \DimaG{real} vector space of dimension $n$. Consider the standard symplectic form on $V\oplus V^*$. Let $C\subset V_{\DimaH{\C}}\oplus V_{\DimaH{\C}}^*$ be a closed conic subvariety of dimension $n$\DimaG{, defined over $\R$}. Let $p:V_{\DimaH{\C}}\oplus V_{\DimaH{\C}}^*\to V_{\DimaH{\C}}$ denote the projection. Then there exists a linear symplectic automorphism $g \in \Sp(V\oplus V^*)\DimaG{(\R)}$ such that $p|_{gC}$ is a finite map.
\end{cor}
\begin{proof}
By \Cref{lem:Lag} there exists a Lagrangian subspace $L\subset V \oplus V^*$ such that $\DimaG{L_{\C}\cap C} = \{0\}$.
 Since the action of $\Sp(V\oplus V^*)\DimaG{(\R)}$ on Lagrangian subspaces is transitive, there exists $g \in \Sp(V\oplus V^*)\DimaG{(\R)}$ such that \DimaC{$V^*=gL$  and thus $\DimaG{gC\cap V_{\C}^*}=\{0\}$.}
From \Cref{lem:PrFin} we get that $p|_{gC}$ is a finite map.
\end{proof}

\begin{proof}[Proof of \Cref{thm:FinSol}]
Let $X=\bigcup_{i=1}^l U_i$ be an open affine cover of $X$ and let $\phi_i:U_i \hookrightarrow \mathbb{A}_{\DimaH{\C}}^{n_i}$ be closed embeddings. Clearly $$\dim \Hom (M,\Sc^*_{\DimaG{X}})\leq \sum_{i=1}^l \dim \Hom (M|_{\DimaG{(U_i)_{\C}}},\Sc^*(U_i\DimaG{(\R))}).$$ By  \Cref{lem:PullDist} $$\Hom (M|_{\DimaG{(U_i)_{\C}}},\Sc^*(U_i\DimaG{(\R)})) \cong \Hom (M|_{\DimaG{(U_i)_{\C}}},\phi_i^!(\Sc^*(\R^{n_i}))\cong\Hom ((\phi_i)_*(M|_{\DimaG{(U_i)_{\C}}}),\Sc^*(\R^{n_i})).$$
Thus it is enough to show that for any holonomic $D$-module $N$ on an affine space $\bA^n_{\DimaH{\C}}$ we have
$$\dim \Hom (N,\Sc^*(\R^{n})) \leq \deg_b(N).$$

Let $C\subset \bA^{2n}_{\DimaH{\C}}$ be the singular support of $N$ with respect to the Bernstein filtration.
By \Cref{cor:PrLag}, there exists $g \in \Sp_{2n}\DimaG{(\R)}$ such that $p|_{gC}$ is a finite map, where $p:\bA_{\DimaH{\C}}^{2n}\to \bA_{\DimaH{\C}}^n$ is the projection on the first $n$ coordinates. By Corollary \ref{cor:Weil}  we have $$\dim \Hom(N,\Sc^*(\R^n)) = \dim \Hom(N^g,\Sc^*(\R^n)^g)=\dim \Hom(N^g,\Sc^*(\R^n)).$$
By Lemma \ref{lem:SympSupp} we have $SS_b(N^g)=gC$. Let $F$ be a good filtration on $N^g$ (with respect to the Bernstein filtration on $D(\bA_{\DimaH{\C}}^n)$). We see that $\Gr N^g$ is finitely generated over $\cO(\bA_{\DimaH{\C}}^n)$, and thus so is $N^g$. Thus $N^g$ is a smooth $D$-module. Note that $\rk_{\cO(\bA^n)} N^g \leq \deg_b N^g=\deg_b N$. By Lemma \ref{lem:Free}  $\dim \Hom(N^g,\Sc^*(\R^n)) \leq \rk_{\cO(\bA_{\DimaH{\C}}^n)} N^g$.
\end{proof}

\subsection{Families of $D$-modules}\label{subsec:DimSolRel}
\Rami{In this section we discuss families of $D$-modules on algebraic varieties over an arbitrary field $k$ of zero characteristic.}
\begin{notn}
Let $\phi:X \to Y$ be a map of algebraic varieties and $\cM$ be a quasi-coherent sheaf of $\cO_X$-modules. For any $y\in Y$, denote by $\cM_y$ the pullback of $\cM$ to $\phi^{-1}(y)$.
\end{notn}

\begin{defn}Let $X,Y$ be smooth algebraic varieties.
\begin{itemize}
\item If $X$ and $Y$ are affine we define the algebra $D(X,Y)$ to be $D(X)\otimes_k \cO(Y)$.
\item Extending this definition we obtain a  sheaf of algebras $D_{X,Y}$ on $X \times Y$.
\item \Rami{By a family of $D_X$-modules parameterized by $Y$, we mean a sheaf of right }modules over the sheaf of algebras $D_{X,Y}$ on $X \times Y$ which is quasicoherent as a sheaf of $\cO_{X\times Y}$-modules.

\item We call a family of $D_X$-modules parameterized by $Y$ coherent if it is locally finitely generated as a $D_{X,Y}$-module.

\item For a family $\cM$ of $D_X$-modules parameterized by $Y$ and a point $y\in Y$, we  call $\cM_y$ the specialization of $\cM$ at $y$ and consider it with the natural structure of a $D_X$-module.
\item We say that a coherent family $\cM$ is holonomic if every specialization is holonomic.
\end{itemize}

\end{defn}

\begin{thm}\label{thm:DegBound}
Let $X,Y$ be smooth algebraic varieties and $\cM$  be a family of $D_X$-modules parametrized by $Y$. Then $\deg\cM_y$ is bounded when $y$ ranges over the $k$-points of $Y$.
\end{thm}

\begin{proof}

Without loss of generality, we can assume that $X=\bA^n$ and $Y$ is an affine variety, and prove that $\deg_b(\cM_y)$ is bounded. We will prove this by induction on $\dim Y$.

The Bernstein filtration on $D(\bA^n)$ gives rise to a filtration on $D(\bA^n,Y)$. Choose a filtration $F$ on $\cM$ which is good with respect to this filtration and let $N:=\Gr{\cM}$, considered as a graded $\cO(\bA^{2n}\times Y)$-module.
Associate to $N$ a coherent sheaf $\cN$ on $\bP^{2n-1}\times Y$.
Let $\cN_y$ be the pullback  of $\cN$ under the embedding of $\bP^{2n-1}$ into $\bP^{2n-1}\times Y$ given by $x \mapsto (x,y)$.
By definition, the Hilbert polynomial of $\cM_y$ with respect to the filtration induced by $F$ is the Hilbert polynomial of $\cN_y$.
\Rami{By Corollary \ref{cor:loc_const}, there exists an open dense subset $U \subset Y$ such that the Hilbert polynomial of $\cN_y$ does not depend on $y$ as long as $y\in U$.}
By the induction hypothesis, $\deg_b(\cM_y)$ is bounded on $Y \setminus U$, and thus bounded on $Y$.
\end{proof}

For an application of this theorem we will need the following lemma.
\DimaD{
\begin{lem}\label{lem:action}
Let a real Lie algebra $\fg$ act on \DimaG{an algebraic manifold $X$ defined over $\R$} and on \DimaG{a complex} algebraic vector bundle $\cE$ on $X$. Fix a natural number $n$ and let $Y$ be the variety of all representations of $\fg$ on $\C^n$. Then there exists a coherent family $\cM$ of $D_X$-modules parameterized by $Y$ such that, for any $\tau \in Y$, we have
\begin{enumerate}
\item $\Hom_{\fg}(\tau,\Sc^*(X(\R),\cE))=\Hom_{D_X}(\cM_{\tau},\Sc^*_X).$
\item  The singular support of $\cM_{\tau}$ (with respect to the geometric filtration) is included in $$\{(x,\phi) \in T^*X \, \vert \,  \forall \alpha \in \fg \text{ we have } \langle \phi, \alpha(x)\rangle =0 \}. $$
\end{enumerate}

\end{lem}
\begin{proof}
It is enough to prove the lemma for affine $X$. Let $N$ be the coherent sheaf of the regular {(algebraic)} sections of $\cE$ (considered as a sheaf of  $\cO_{\DimaG{X_{\C}}}$-modules). Let $\cN$ be the pullback of $N$ to $\DimaG{X_{\C}\times Y}$. Let $\cN':=\cN\otimes_{\cO_{\DimaG{X_{\C}\times Y}}}D_{\DimaG{X_{\C}, Y}}\otimes_{\C}\C^n$, and  $\cN''\subset \cN'$ be the \Rami{$D_{\DimaG{X_{\C}, Y}}$-submodule generated} by elements of the form $$\alpha n \otimes 1 \otimes v+ n \otimes \xi_{\alpha} \otimes v + n\otimes  f_{\alpha}(v),$$
where $\alpha\in \fg$, $\xi_{\alpha}$  is the vector field on $X$ corresponding to $\alpha$,  and $f_{\alpha}(v)\in D_{\DimaG{X_{\C}, Y_{\C}}}\otimes_{\C}\C^n$ is the $\C^n$-valued regular function on $\DimaG{X_{\C}\times Y}$ given by $f_{\alpha}(v)(x,\tau)=\tau(\alpha)v$. Then $\cM:=\cN'/\cN''$ satisfies the requirements.\end{proof}
}
\Cref{thm:FinSol,thm:DegBound,lem:action} imply \Cref{thm:FinOrb}.

\section{Proof of \Cref{thm:main,thm:IntroMult}}\label{sec:Rep}

In this section, we derive  \Cref{thm:main,thm:IntroMult}  from Theorem \ref{thm:geo} and \S \ref{sec:DimSol}. We do that by embedding the multiplicity space into a certain space of \DimaD{relative} characters.

\subsection{Preliminaries}

For  \DimaG{a reductive group $G$ defined over $\R$}, we denote by $Irr(G\DimaG{(\R)})$ the collection of irreducible admissible smooth \Fre representation of $G\DimaG{(\R)}$ of moderate growth. We refer to \cite{CasGlob,Wal} for the background on these representations.

\begin{theorem}[See {\cite[Theorem 4.2.1]{Wal}}]\label{thm:CentFin}
\Rami{The center $\fz(\cU(\fg))$ of the universal enveloping algebra of the \DimaG{complexified} Lie algebra of $G$ acts finitely on every admissible smooth \Fre representation $\pi$ of $G$ of moderate growth.
This means that there exists an ideal in $ \fz(\cU(\fg))$ of finite codimension that annihilates $\pi$.
}
\end{theorem}


\begin{lem}[{\cite[Theorem 1.2 and Corollary 1.4]{Adams}}]\label{lem:GK}
For any \DimaG{reductive group $G$ defined over $\R$}, there exists an involution $\theta$ of $G$ such that, for any $\pi\in \Irr(G\DimaG{(\R)})$, we have $\hat \pi \cong \pi^{\theta}$. 
\end{lem}

\begin{thm}[Casselman embedding theorem, see {\cite[Proposition 8.23]{CM}}]\label{thm:CasSubRep}
Let $G$ be \DimaG{a reductive group $G$ defined over $\R$} and  $P$ be a minimal parabolic subgroup of $G$. Let $\pi\in Irr(G\DimaG{(\R)})$.   Then there exists a
finite-dimensional
representation $\sigma$ of $P$ and an epimorphism $Ind_{P\DimaG{(\R)}}^{G\DimaG{(\R)}}(\sigma) \twoheadrightarrow \pi$.
\end{thm}

\subsection{Proof of Theorem \ref{thm:main} and Proposition \ref{prop:S}}\label{subsec:PfMain}

Theorem \ref{thm:main} follows from \Cref{thm:geo} and Proposition \ref{prop:S}.
{\sloppy
\begin{proof}[Proof of Proposition \ref{prop:S}]
Let $\xi$ be a \DimaD{relative} character of a smooth admissible \Fre  representation $\pi$ of $G\DimaG{(\R)}$ of moderate growth with respect to a pair of subgroups $(H_1\DimaG{(\R)},H_2\DimaG{(\R)})$ and their characters $\chi_1,\chi_2$. By \Cref{thm:CentFin}, \Rami{  there exists an ideal $I \subset \fz(\cU(\fg))$ of finite codimension that annihilates $\pi$ and thus  annihilates $\xi$.} For any element $z \in \fz(\cU(\fg))$, there exists a polynomial $p$ such that $p(z)\in I$ and thus $p(z)\xi=0$. This implies that the symbol of any  $z \in \fz(\cU(\fg))$ of positive degree vanishes on the singular support of $\xi$. It is well-known  that
the joint  zero-set of these symbols over each point $g\in G$ is the nilpotent cone $\cN(\fg^*)$. Since $\xi$ is $(\DimaC{\fh}_1\times \DimaC{\fh}_2, \chi_1\times \chi_2)$-equivariant,
this implies that the singular support of $\xi$ lies in $S$.
\end{proof}}
\subsection{Proof of Theorem \ref{thm:IntroMult}}\label{subsec:PfIntroMult}
%
%
Part \eqref{it:Fin} follows immediately from \Cref{thm:FinOrb} and the Casselman embedding theorem. If $G$ is quasi-split then so does part \eqref{it:Bound}.
For the proof of part \eqref{it:Bound} in the general case, we will need the following lemma.

\begin{lem}\label{lem:SphCharBound}
\Rami{Let $G$ be \DimaG{a reductive group $G$ defined over $\R$}  and $H_1,H_2$ be spherical subgroups. Let $Y=\mathrm{Spec}(\fz(\cU(\fg)))\Andrey{\times Y_1\times Y_2}$, \Andrey{where $Y_i$ is the variety of characters of $\ha_i=\mathrm{Lie\,}H_i$}. For any $\lambda\in Y(\C)$, define $U_{\lambda,\Andrey{\chi_1,\chi_2}}:=\Sc^*(G\DimaG{(\R)})^{\Andrey{\ha_1\times \ha_2,(\chi_1,\chi_2)},(\fz(\cU(\fg)),\lambda)}$ to be the space of tempered distributions on $G$ that are left \Andrey{$\chi_1$-equivariant with respect to $\ha_1$, right $\chi_2$-equivariant with respect to $\ha_2$} and are eigendistributions with respect to the action of $\fz(\cU(\fg))$ with eigencharacter $\lambda$. Then $\dim  U_{\lambda,\chi_1,\chi_2}$ is bounded over $Y(\C)$}.
\end{lem}

\begin{proof}
\Rami{
Let us construct a family of $D(G)$-modules $\cM$ parameterized by $Y$. For any $\alpha \in \fg$, let $r_\alpha$ and $l_\alpha$ be the corresponding right and left invariant vector fields on $G$ considered as  elements in  $D(G,Y)$. For any $\beta \in \fz(\cU(\fg))$, \Andrey{$\alpha_i\in\ha_i$, let $f_\beta,g^i_{\alpha_i}$ be the functions on $Y$ that send $(\mu,\gamma_1,\gamma_2)\in Y$ to $\mu(\beta), \gamma_i(\alpha_i)$ respectively}. Let also $d_\beta$ be the differential operator on $G$ corresponding to $\beta$, \DimaC{such that
$d_{\beta}\xi=\beta\xi$ for any distribution $\xi$ on $G\DimaG{(\R)}$}.
We consider \Andrey{$d_\beta,r_{\alpha_1},l_{\alpha_2},f_\beta,g^{i}_{\alpha_i}$} as elements of $D(G,Y)$. Let $I \subset D(G,Y)$ be the ideal generated by \DimaC{$r_{\alpha_1}\Andrey{-g^{1}_{\alpha_1}}$, $l_{\alpha_2}\Andrey{-g^{2}_{\alpha_2}}$ and $f_\beta \DimaC{-} d_\beta$ where $\alpha_i\in\fh_i$} and $\beta\in \fz(\cU(\fg))$. Define $\cM:=D(G,Y)/I$.

It is easy to see that $U_{\lambda,\chi_1,\chi_2} \cong \Hom(\cM_{(\lambda,\chi_1,\chi_2)},\Sc^*(G\DimaG{(\R)}))$. As in the proof of Proposition \ref{prop:S}, the singular support of $\cM_{(\lambda,\chi_1,\chi_2)}$ lies in $S$, for any $\lambda,\chi_1,\chi_2$. By Theorem \ref{thm:geo}, $\cM_{(\lambda,\chi_1,\chi_2)}$ is holonomic and, therefore, $\cM$ is holonomic. By Theorem \ref{thm:FinSol}, we have $\dim U_{\lambda,\chi_1,\chi_2} \leq \deg \cM_{(\lambda,\chi_1,\chi_2)}$. By \Cref{thm:DegBound}, $\deg \cM_{(\lambda,\chi_1,\chi_2)}$ are bounded.}
\end{proof}


\begin{proof}[Proof of Theorem \ref{thm:IntroMult}\eqref{it:Bound}]

We choose an involution $\theta$ as in \Cref{lem:GK}, let $H_1:=H, \, H_2:=\theta(H)$, and  define the spaces $U_{\lambda}$ as in \Cref{lem:SphCharBound}.

Now let $\pi\in \Irr(G\DimaG{(\R)})$ \Andrey{and let $\chi$ be a character of $\fh$} such that $(\pi^*)^{\fh,\chi}\neq 0$. Let ${\lambda}$ stand for the infinitesimal character of $\pi$. By \Cref{lem:GK}, $(\hat \pi^*)^{d\theta(\fh),\Andrey{d\theta(\chi)}}\neq 0$. Fix a non-zero $\phi \in (\hat\pi^*)^{{d\theta(\fh),\Andrey{d\theta(\chi)}}}$. Then $\phi$ defines an embedding $(\pi^*)^{\fh,\Andrey{\chi}} \hookrightarrow U_{\lambda,\Andrey{\chi,d\theta(\chi)}}$ by $\psi \mapsto \xi_{\psi,{\phi}}$, where $\xi_{\psi,{\phi}}$ is the \DimaD{relative} character, which is defined by $\xi_{\psi,{\phi}}(f):=\langle \psi , \pi(f) \phi\rangle$. Thus, $\dim (\pi^*)^{\fh,\Andrey{\chi}} \leq \dim U_{\lambda,\Andrey{\chi,d\theta(\chi)}}$, which is bounded by \Cref{lem:SphCharBound}.
\end{proof}

\appendix

\section{Proof of Lemma \ref{lem:PullDist}}\label{app:dist}
\setcounter{theorem}{0}
For the proof, we will need the following standard lemmas.
Let $M$ be a smooth manifold and $N \subset M$ be a closed smooth submanifold.

\begin{lemma}\label{lem:id}
 Denote $I_N:=\{f \in \Cc(M) \, \vert f|_N=0\}$.
Let $J\subset I_N$ be an ideal in $\Cc(M)$ such that
\begin{enumerate}
\item For any $x\in N$, the space $\{d_xf\,|\,f\in J\}$ is the conormal space to $N$ in  $M$ at the point $x$.
\item For any $x \in M \setminus N$, there exists $f \in J$ such that $f(x)\neq 0$.
\end{enumerate}
Then $J=I_N$.
\end{lemma}
\begin{proof} Using partition of unity, it is enough to show that, for any $f\in I_N$ and $x \in M$, there exists $f' \in J$ such that $f$ coincides with $f'$ in a neighborhood of $x$. For $x \notin N$ this is obvious, so we assume that $x \in N$. We prove the statement by induction on the codimension $d$ of $N$ in $M$. The base case $d=1$ follows, using the implicit function theorem, from  the case $N=\R^{n-1}\subset \R^{n}=M$, which is obvious.

For the induction step, take {an element} $g\in J$ such that $d_xg\neq 0$. Let $$Z:=\{y \in M \, \vert \, g(y)=0\} \text{ and }U:=\{y\in M \, \vert \,d_yg\neq 0 \}. $$ By the implicit function theorem,  $U \cap Z $ is a closed submanifold of $U$. Choose $\rho\in \Cc(M)$ such that $\rho=1$ in a neighborhood of $x$ and $\Supp(\rho) \subset U$. Let $\bar f:=(\rho f)|_{U \cap Z}$. Let $$\bar J := \{\alpha |_{U \cap Z} \vert \alpha \in J \text{ and } \Supp \alpha \subset U\}.$$ By the induction hypothesis, $\bar f \in \bar J.$  Thus, there exists $f'' \in J$ such that $f-f''$ vanishes in a neighborhood of $x$ in $Z$. Now, the case $d=1$ implies that  there exists $\alpha \in \Cc(M)$ such that $f-f''$ coincides with $\alpha g$ in a neighborhood of $x$.
\end{proof}

\begin{lemma}\label{lem:OpenMap}
The restriction $\Cc(M) \to \Cc(N)$ is an open map.
\end{lemma}
\begin{proof}
Let $K\subset M$ be a compact subset. It is easy to see that there exists a compact $K' \supset K$ such that the restriction map $\mathrm{C}^{\infty}_{K'}(M)\to \mathrm{C}^{\infty}_{K'\cap N}(N)$ is onto, using the partition of unity. By  the Banach open map theorem this map is open. Thus, the restriction $\Cc(M) \to \Cc(N)$ is an open map.
\end{proof}

Let $Y$ be \DimaG{an affine algebraic manifold defined over $\R$} and $X$ be a closed algebraic submanifold. Let $i:X \to Y$ denote the embedding.

\begin{lem}\label{lem:Emb}
Let $\xi$ be a  distribution on $X\DimaG{(\R)}$ such that $i_*\xi$ is a tempered distribution. Then $\xi$ is a tempered distribution.
\end{lem}
\begin{proof}
The map $i_*$ is dual to the pullback map $C_c^{\infty}(Y\DimaG{(\R)})\to C_c^{\infty}(X\DimaG{(\R)})$. This can be extended to a continuous map $i^*:\Sc(Y\DimaG{(\R)})\to \Sc(X\DimaG{(\R)})$ which is onto by \cite[Theorem 4.6.1]{AGSc}. The Banach open map theorem implies that $i^*$ is an open map. It is easy to see that $i_{*}\xi:\Sc(Y\DimaG{(\R)})\to \C$ vanishes on $Ker(i^*)$, and thus it gives rise to a continuous map $\Sc(X\DimaG{(\R)})\to \C$, which extends $\xi$.
\end{proof}

\begin{lem}\label{lem:DistPush}
Let $ \xi$ be a \DimaG{complex valued} distribution on $Y\DimaG{(\R)}$ such that $p\xi=0$ for any polynomial $p$ on $Y$ that vanishes on $X$. Then $\xi$ is a pushforward of a distribution on $X\DimaG{(\R)}$.
\end{lem}
\begin{proof}
Let $J(X)$ be the ideal of all polynomials on $Y$ that vanish on $X$. Let $J :=J(X) \Cc(Y\DimaG{(\R)})$. By Lemma \ref{lem:id} we have $J=I_{X\DimaG{(\R)}}$. Thus, $\xi$ vanishes on $I_{X\DimaG{(\R)}}$ and thus, by Lemma \ref{lem:OpenMap}, $\xi$ is a pushforward of a distribution on $X\DimaG{(\R)}$.
\end{proof}

Lemma \ref{lem:PullDist} follows from \Cref{lem:Emb,lem:DistPush}  and the definition of $i^!$ for  closed embedding of smooth affine varieties. 

\sloppy
\printbibliography
\end{document}